\documentclass[12pt,oneside]{amsart}
\usepackage{amsmath,amssymb,amsthm,color,graphicx}
\usepackage{accents}
\usepackage[unicode,breaklinks=true,colorlinks=true]{hyperref}
\usepackage[top=1in, bottom=1in, left=1.25in, right=1.25in]{geometry}

\numberwithin{equation}{section}
\newtheorem{theorem}{Theorem}[section]
\newtheorem*{theorem*}{Theorem}

\newtheorem{lemma}[theorem]{Lemma}

\newtheorem{definition}[theorem]{Definition}

\theoremstyle{remark}
\newtheorem{remark}[theorem]{Remark}
\newtheorem*{remark*}{Remark}

\newcommand{\bke}[1]{\left ( #1 \right )}
\newcommand{\bkt}[1]{\left [ #1 \right ]}
\newcommand{\bket}[1]{\left \{ #1 \right \}}
\newcommand{\norm}[1]{ \| #1  \|}

\newcommand{\abs}[1]{\left | #1 \right |}

\newcommand\al{\alpha}

\newcommand\de{\delta}
\newcommand\ep{\epsilon}

\newcommand\e {\epsilon}

\newcommand\la{\lambda}

\newcommand{\R}{\mathbb{R}}
\newcommand{\RR}{\mathbb{R}}

\renewcommand{\div}{\mathop{\rm div}}

\newcommand{\esssup} {\mathop{\rm ess\,sup}}
\newcommand{\CKN}{\text{\tiny CKN}}

\newcommand{\EQ}[1]{\begin{equation} #1 \end{equation}}
\newcommand{\EQS}[1]{\begin{equation}\begin{split} #1 \end{split}\end{equation}}
\newcommand{\EQN}[1]{\begin{equation*}\begin{split} #1 \end{split}\end{equation*}}

\newcommand{\loc}{\mathrm{loc}}
\newcommand{\uloc}{\mathrm{uloc}}

\newcommand*{\dt}[1] {\accentset{\mbox{\large\bfseries .}}{#1}}
\renewcommand{\dot}{\dt}

\begin{document}

\title[Local regularity conditions for Navier-Stokes equations]
{Local regularity conditions on initial data for local energy solutions of the Navier-Stokes equations}
\author{Kyungkeun Kang}
\address[K. Kang]{Department of Mathematics, Yonsei University, Seoul 120-749, South Korea} 
\email{kkang@yonsei.ac.kr}

\author{Hideyuki Miura}
\address[H. Miura]{Department of Mathematical and Computing Science, Tokyo Institute of Technology, Tokyo 152-8551, Japan}
\email{miura@is.titech.ac.jp}

\author{Tai-Peng Tsai}
\address[T.-P. Tsai]{Department of Mathematics, University of British Columbia,
Vancouver, BC V6T 1Z2, Canada}
\email{ttsai@math.ubc.ca}

\thanks{
We thank Professor Reinhard Farwig for valuable suggestions 
on Theorem \ref{thm:0430}.
The research of Kang was partially supported by NRF-2019R1A2C1084685 and NRF-2015R1A5A1009350. The research of Miura was partially supported
by JSPS grant 17K05312. The research of Tsai was partially supported
by NSERC grant RGPIN-2018-04137.}

\maketitle

\begin{center}
Dedicated to Hideo Kozono on the occasion of his 60th birthday
\end{center}

\begin{abstract}
We study the regular sets of local energy solutions
to the Navier-Stokes equations in terms of conditions on the initial data.
It is shown that if a weighted $L^2$ norm of the initial data is finite, then all local energy solutions are regular in a region
confined by space-time hypersurfaces determined by the weight.
This result refines and generalizes Theorems C and D of Caffarelli, Kohn and Nirenberg (Comm. Pure Appl. Math. 35; 1982) and  our recent paper \cite{KMT21} as well.

\end{abstract}

\section{Introduction}
\subsection{Regular sets for local energy  solutions}

We are concerned with the regularity of weak solutions of the incompressible Navier-Stokes equations
\begin{align}\label{NS}\tag{\textsc{ns}}
\partial_tv-\Delta v +v\cdot\nabla v+\nabla p = 0,
\quad
\div v =0
\end{align}
in $\R^3$,
associated with the initial value $v|_{t=0}=v_0$ with $\div v_0=0$.
 The global existence of weak solutions for finite energy
initial data was established by Leray \cite{Leray},  and extended to domains by Hopf \cite{Hopf}.
Despite a lot of efforts since their foundational work, the global regularity of the weak solutions remains a longstanding open problem.
It is known from works \cite{FJR,  Kato84, GM} that for $v_0 \in L^q(\R^3)$ with $q \ge 3$, \eqref{NS} has a regular solution defined on some short time interval where it satisfies the following estimate:
\begin{equation}
\label{infty}
\|v(t)\|_{L^\infty} \le C(v_0)t^{-\frac3{2q}}
\end{equation}
with the constant $C(v_0)=C_q\|v_0\|_{L^q}$.
 Motivated by the problem for constructing large forward self-similar solutions to \eqref{NS},
Jia and \v Sver\'ak \cite{JS} asked under which condition this result can be localized in space. 
For $B_r(x)=\{y \in \R^3:\, |x-y|<r\}$ and $B_r=B_r(0)$, 
 their question can be stated as follows:
\begin{enumerate}
\item[(R)]
If $v_0$ is a general initial data for
which a suitable weak solution $v$ is defined and $v_0|_{B_2} \in L^q(B_2)$,
can we conclude that $v$ is regular in $B_1 \times (0, t_1)$ for some time $t_1 > 0$?
\end{enumerate}
This question is settled affirmatively for the scale subcritical case $q>3$ in \cite{JS} and for the critical case $q=3$ 
in \cite{BP,KMT20}; see also \cite{Tao} for a condition on the initial enstrophy and
\cite{BP} for further extension to the $L^{3,\infty}$ space and the critical Besov spaces.
In our previous work \cite{KMT21}, we revisited this problem from a slightly different view point.
Let us define the scaled local energy of the initial data
by
\begin{equation}
\label{N}
 N_{(\alpha)}(v_0):= \sup_{r\in (0,1]} \frac{1}{r^{\alpha}}\int_{B_r} \abs{v_0(x)}^2dx
\end{equation}
for $\alpha >0$.
We showed that if $N_{(1)}$ (with $\al=1$)
is sufficiently small, any local energy solution is regular and
satisfies the estimate \eqref{infty} in the region
\[
\bket{(x,t) \in B_1\times (0,\infty)\,: cN_{(1)}^2\abs{x}^2 \le t < t_1}
\]
for some $t_1>0$;\,see Section 2 for the definition of the local energy solutions. 
Concerning the question (R),
the above result shows that if the local Morrey norm defined by 
$$
\|f\|_{m^{2,1}(B_2)} :=\sup_{x_0 \in B_2, \,r\in (0,1]}
\left( \frac{1}{r}\int_{B_r(x_0) \cap B_2} \abs{f(x)}^2dx \right)^{\frac12}
$$
is sufficiently small, the local energy solution is regular in $B_1 \times (0,t_1)$
for some $t_1>0$ and satisfies the $L^\infty$ 
bound \eqref{infty}.
By the well-known relation $m^{2,1} \hookleftarrow L^3$, the latter recovers the results
on the local regularity results \cite{BP, KMT20} for  $L^3$ data. 
The goal of the present work is to study the local in space regularity 
of the solutions with data satisfying the condition 
$N_{(\alpha)} <\infty$ for general $\alpha >1$.
Intuitively, the bigger $\alpha$ is, the more regular the data is, and so 
is the solution at least locally. We justify it in the following sense:
The solution is regular in a larger region and it also satisfies 
an improved $L^\infty$ estimate in a slightly different region.
More precisely, it can be stated as follows.

\begin{theorem}
\label{thm:main}
Assume that a local energy solution $v$ in $\R^3 \times (0, T)$
with the initial data $v_0 \in L^2_{\uloc}$ satisfies 
$N_{(\alpha)} <\infty$ for some $\alpha >1$.
Then there exists $T_1=T_1(\|v_0\|_{L^2_{\uloc}}, N_{(\alpha)}, \al, T)
>0$ 
such that $v$ is regular in the region
\[
\Pi_1=\bket{(x,t) \in B_{\frac12} \times (0,T)\,:
 c_1N_{(\alpha)}^2\abs{x}^{2\alpha}\le t<T_1}
\]
and satisfies
\begin{equation}
\label{0204}
\abs{v(x,t)}\le Ct^{-\frac{1}{2}} \qquad  \mbox{for} \ (x,t) \in 
\Pi_1.
\end{equation} 
Moreover, if $\alpha \in (1,4)$, 
there exist $T_2=T_2(\|v_0\|_{L^2_{uloc}}, N_{(\al)}, \al, T)>0$
and $M=M(\|v_0\|_{L^2_{uloc}}, N_{(\al)}, \al)$$>0$ such that 
\begin{equation}
\label{0205}
|v(x,t)| \le Mt^{\frac{\alpha-3}{4}} \qquad \textit{for} \ (x,t) \in \Pi_2
\end{equation}
holds with
\[
\Pi_2=\bket{(x,t) \in B_{\frac12} \times (0,T)\,:
 \frac{c_2N_{(\al)}^{\frac{2}{\alpha}}}{(1+N_{(\al)})^{2+\frac{2}{\al}}}\abs{x}^{2}\le t<T_2}.
\]
 Here  $c_1$, $c_2$, and $C$ 
are positive constants depending only on $\al$.
\end{theorem}

\begin{figure}[h!]
\label{fig1.1}
\setlength{\unitlength}{2pt}
\begin{picture}(160,60)
\thicklines
\put(80,0){\vector(0,1){58}}
\put(80,0){\vector(1,0){90}}
\put(80,0){\vector(-1,0){90}}
\qbezier[2000](30,50)(80,-50)(130,50)
\qbezier[2000](15,25)(80,-25)(145,25)
\put(0,50){\line(1,0){160}}
\put(0,25){\line(1,0){160}}
\put(0,-3){\line(0,1){6}}
\put(160,-3){\line(0,1){6}}
\put(158,4){{\tiny $1/2$}}
\put(172,0){$x$}
\put(79,60){$t$}
\put(73,27){$T_1$}
\put(73,52){$T_2$}
\put(55,35){$\Pi_2$}
\put(41,12){$\Pi_1$}
\put(125,35){$t=c|x|^2$}
\put(134,12){$t=c|x|^{2\al}$}
\end{picture}
\setlength{\unitlength}{1pt}
\caption{Regular regions for Theorem \ref{thm:main} for the 
case $T_2 > T_1$. }
\end{figure}
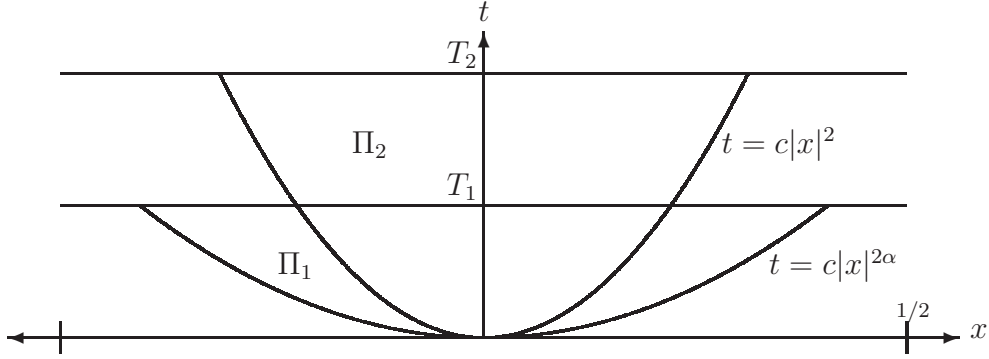

\emph{Comments for Theorem \ref{thm:main}:}

\begin{enumerate}
\item[a)] Theorem \ref{thm:main} is also valid for $\al=1$ if we further assume smallness of $N_{(1)}$. In that case we recover Theorem 3.4 of \cite{KMT21}.
\item[b)] In Theorem \ref{large data} we will give a more general statement for $\al \in (1,4)$ in terms of 
\[
N_{(\alpha),R}:= 
\sup_{r \in (R, 1]} \frac{1}{r^\alpha}\int_{B_r(0)}|v_0|^2
\qquad
\Big(0 \le R\le \frac12 \Big).
\] 

\item[c)] $T_1$ will be given by \eqref{T1.def}. $T_2$ will be given by \eqref{T2.def} as $S_2$ with $N_R = N_{(\al),0}$. 
None is necessarily larger than the other.
\end{enumerate}

As one of applications of our result to some specific cases, 
we show a regularizing estimate of the form \eqref{infty}
 for locally $L^q$ data, $3<q\le \infty$. 

\begin{theorem}
\label{thm:0209}
Let $(v,p)$ be a local energy solution in $\R^3 \times (0, T)$
 with the initial data $v_0 \in L^2_{\uloc}$. 
If $v_0 \in L^q(B_2)$ for some $q \in (3,\infty]$, then 
there exists $T_3=T_3(q, \|v_0\|_{L^q(B_2)}, \|v_0\|_{L^2_{\uloc}})>0$ 
such that $v$ is regular in $B_{1} \times (0,T_3)$
and satisfies
\begin{equation}\label{th1.2-eq1}
\|v(t)\|_{L^\infty(B_1)} \le C
(1+\|v_0\|_{L^q(B_2)})^{3-\frac6{q}}\|v_0\|_{L^q(B_2)}t^{-\frac{3}{2q}} \qquad \textit{for} \ t \in (0,T_3)
\end{equation}
with an absolute constant $C>0$.
\end{theorem}
\noindent
Estimate \eqref{th1.2-eq1}  does not explicitly depend on $\|v_0\|_{L^2_{\uloc}}$. Its dependence on $\|v_0\|_{L^2_{\uloc}}$ is through the time upper bound $T_3$.
In \cite[Theorem 3]{BP2}, Barker and Prange obtained
an $L^\infty$ bound in some time interval for locally 
$L^6$ initial data with the aim of the behavior of the
$L^3$ norm near the singular point.

In \cite{CKN}, Caffarelli, Kohn, and Nirenberg established 
estimates of the regular sets for the suitable weak solutions for data in 
$L^2$ weighted spaces. 
Define weighted $L^2$ norm
\EQ{
\|v_0\|_{L^{2,\beta}}:=\||x|^{\frac{\beta}{2}}v_0\|_{L^2(\R^3)}
}
 for $\beta \in \R$.
It is shown in \cite{CKN} that if the finite energy data satisfies $\|v_0\|_{L^{2,1}} <\infty$,
then there exists a suitable weak solution which is regular in the set 
of points satisfying $t> \min(C,K|x|^{-2})$.
Moreover it is also shown that if the finite energy data satisfies $\|v_0\|_{L^{2,-1}} <\epsilon_0$
with some absolute small constant $\epsilon_0>0$, then
there exists a suitable weak solution which is regular in the set $t>C |x|^2$; see \cite{DL} for its refinement.
In \cite[Corollary 1.4]{KMT21}, we extended estimates of the regular sets for general data in $L^{2,\beta}$, $\beta \ge  -1$.
By using Theorem \ref{thm:main}, we are able to show similar
 results  for $\beta$ below $-1$ at least locally in time.

\begin{theorem}
\label{thm:0430}
Let $(v,p)$ be a local energy solution in $\R^3 \times (0,T)$
 for  the initial data $v_0 \in L^2_{\uloc}$.
 Assume that $v_0 \in L^{2,\beta}(\R^3)$ for some $\beta <-1$.
 Then there exist
positive constants $T_4=T_4(v_0)$ and $c(v_0)$ such that $v$ is regular in the set
$$
\bket{(x,t) \in \R^3 \times (0,T)\,:\,c(v_0)|x|^{-2\beta}\le t<T_4(v_0)}.
$$
\end{theorem}

\medskip

The proofs of Theorems \ref{thm:main}, \ref{thm:0209} and \ref{thm:0430} are based on a local-in-space a priori estimate for the scaled local energy %
defined by
\begin{equation}
\label{scaled energy}
E_r(t):=\esssup_{0<s<t}\frac{1}{r^{\alpha}}\int_{B_r} \abs{v(s)}^2+\frac{1}{r^{\alpha}}\int_0^t\int_{B_r} \abs{\nabla v}^2,
\end{equation}
which is useful to focus on the local regularity at the origin
and is used in \cite{KMT21} for the case $\alpha=1$.
In order to deal with the non-local effect of the pressure term, 
we apply the decomposition formula of \cite{LR,KS} to the local energy estimate.
Our key observation is that the (global) $L^2_{\uloc}$ norm of the initial data is sufficient to control the non-local effect 
at least locally in time; see Lemma \ref{lem:main}.
It should be noted that our strategy is different from
previous works \cite{JS, BP,KMT20} and that ours provides a rather direct and transparent approach.

\subsection{Outline of the paper and notation}

In section \ref{preliminaries}, we introduce the notion of local energy  solutions and 
recall the local regularity criterion due to \cite{CKN} 
as well as some technical lemmas.
Section \ref{main results} is devoted to stating and proving our main results including Theorem \ref{thm:main}.

Throughout this paper, $C \in (0,\infty)$ 
denotes an absolute constant which may change line by line.

\section{Preliminaries}
\label{preliminaries}
In this section, we recall some notions about the weak solution to
\eqref{NS} and some results such as the $\epsilon$-regularity theorems
and a priori estimates for the solutions.

For any domain $\Omega \subset \R^3$ and open interval
$I \subset (0,\infty)$, we say $(v,p)$ is a
suitable weak solution in $\Omega \times I$
if
it satisfies \eqref{NS} in the sense of distributions in
$\Omega \times I$,
\begin{align*}
v \in L^\infty(I; L^2(\Omega)) \cap L^2(I; \dot{H}^1(\Omega)), \quad  p \in L^{3/2}(\Omega \times I),
\end{align*}
and the local energy inequality:
\EQS{\label{CKN-LEI}
&\int_{\Omega} |v(t)|^2\phi(t) \,dx +2\int_0^t\!\! \int_{\Omega} |\nabla v|^2\phi\,dx\,dt
\\&\leq %
\int_0^t\!\!\int_{\Omega} |v|^2(\partial_t \phi + \Delta\phi )\,dx\,dt +\int_0^t\!\!\int_{\Omega} (|v|^2+2p)(v\cdot \nabla\phi)\,dx\,dt
}
for  all non-negative $\phi\in C_c^\infty ( \Omega \times I)$.
Note that no boundary condition is required.

We next define the notion of local energy solutions,
The following definition is formulated in \cite{BT19}, which is slightly revised
from the notions of the local Leray solution defined in \cite{LR},
the local energy solution in \cite{KS} and the Leray solution in \cite{JS}.
We refer to \cite[Section 2]{KMT20} for discussion of their relation.

\begin{definition}[Local energy solutions \cite{BT19}]\label{def:localenergy}
A vector field $v\in L^2_{\loc}(\R^3\times [0,T))$ is a local energy solution to \eqref{NS} with divergence free initial data $v_0\in L^2_{\uloc}$ if
\begin{enumerate}
\item for some $p\in L^{3/2}_{\loc}(\R^3\times [0,T))$, the pair $(v,p)$ is a distributional solution to \eqref{NS},

\item for any $R>0$,

\begin{equation}\label{uniform-energy}
\esssup_{0\leq t<\min \{R^2,T \}}\,\sup_{x_0\in \R^3}\, \int_{B_R(x_0 )} |v|^2\,dx
+ \sup_{x_0\in \R^3}\int_0^{\min\{ R^2, T\}}\!\!\!\int_{B_R(x_0)} |\nabla v|^2\,dx \,dt<\infty,
\end{equation}

\item for all compact subsets $K$ of $\R^3$ we have $v(t)\to v_0$ in $L^2(K)$ as $t\to 0^+$,

\item $(v,p)$ satisfies the local energy inequality
\eqref{CKN-LEI}
for  all non-negative functions $\phi\in C_c^\infty (Q)$
with all cylinder $Q$ compactly supported in $\R^3 \times (0,T)$,
\item for every $x_0\in \R^3$ and $r>0$, there exists $c_{x_0,r}\in L^{3/2}(0,T)$ such that
\EQS{\label{decomposition}
        p (x,t)-c_{x_0,r}(t)&=\frac 1 3 |v(x,t)|^2 + \mbox{p.v.}\int_{B_{3r}(x_0)} K(x-y):v(y,t)\otimes v(y,t)\,dy
        \\&+\int_{\R^3\setminus B_{3r}(x_0)} (K(x-y)-K(x_0-y)):v(y,t)\otimes v(y,t)\,dy
}
in $L^{3/2}(B_{2r}(x_0) \times (0,T))$, where $K(x)= \nabla^2(\frac{1}{4\pi|x|})$, and

\item for any compact supported functions $w \in L^2(\R^3)^3$,
\EQ{\label{weak-continuity}
\text{the function}\quad
t \mapsto \int_{\R^3} v(x,t)\cdot w(x)\,dx \quad \text{is continuous on }[0,T).
}

\end{enumerate}
\end{definition}

Let us also recall the uniformly local $L^q$ spaces for $1\le q <\infty$.
We say $f\in L^q_{\uloc}$ if $f\in L^q_\loc(\R^3)$ and
\EQ{\label{Lq-uloc}
\norm{f}_{L^q_{\uloc}} =\sup_{x \in\R^3} \norm{f}_{L^q(B_1(x))}<\infty.
}
\medskip
Local in time existence of local energy solutions for initial data in $L^2_{\uloc}$
and also global existence for initial data in $E^2:=\overline{C_0^\infty}^{L^2_{\uloc}}$ and the weighted $L^2$ spaces
are established in \cite{LR} and \cite{FL, BKT}, respectively. 
One of the advantage of the local energy solution
is that it can be defined even for infinite energy data; 
see  \cite{KT, LR2, JST} and references therein for further  developments and its applications,
and \cite{MMP} for the local energy solutions in the half space.

The next lemma shows  a priori bounds for the
local energy solution, where the crucial part is proved by Lemari\'e-Rieusset~\cite{LR}
and later revised in \cite{KS,JS13}.
The following version is deduced from \cite{KMT20, BT19}.
\begin{lemma}[a priori bound of local energy solutions]
 \label{apriori uloc}
Suppose  $(v,p)$ is a local energy solution to \eqref{NS} defined in 
$\R^3 \times (0,T)$
with divergence free initial data $v_0\in L^2_{\uloc}$.
For any $s,q>1$ with   $\frac 2{s}+ \frac 3{q} = 3$,
there exist  $C(s,q)>0$ and $c_{x_0}(t) \in \R^3$
such that
\begin{align}
\label{ineq.apriorilocal}
E_{\uloc}(t):=\esssup_{0\leq \tau \leq t}
\|v(\tau)\|_{L^2_{\uloc}}^2
+
\sup_{x_0\in \RR^3}\int_0^t\!\int_{B_1(x_0)} |\nabla  v|^2 &\le
2 \|v_0\|_{L^2_{\uloc}}^2,
\\
\label{p.apriori}
\sup_{x_0 \in \R^3}  \norm{p -c_{x_0}}_{L^s(0,t; L^q(B_1(x_0)))}
&\le C(s,q) \|v_0\|_{L^2_{\uloc}}^2
\end{align}
for $t\le T_{\uloc} :=\min \bke 
{T, \frac{c_0}{1+\|v_0\|_{L^2_{\uloc}}^4}}
$
with a universal constant $c_0>0$.
Similar estimates with $B_1$ replaced by $B_r$ are valid.
\end{lemma}
We now recall the scaled version of the $\epsilon$-regularity theorem of Caffarelli-Kohn-Nirenberg \cite[Proposition 1]{CKN}. It is formulated in the present form in \cite{NRS,Lin}.
\begin{lemma}\label{CKN-Prop1}
There are absolute constants $\e_{\CKN}$ and $C_{\CKN}>0$ with the following property. Suppose that $(v,p)$ is a suitable weak solution of (NS)
 in $B_r(x_0) \times (t_0-r^2,t_0))$, $r>0$, with
\[
\frac 1{r^{2}} \int^{t_0}_{t_0-r^2} \int_{B_{r}(x_0)} |v|^3
+ |p|^{3/2} dx\,dt \le \e
\]
for some $\e \le \e_{\CKN}$,
then $v  \in L^\infty(B_{\frac{r}{2}(x_0)} \times (t_0-\frac{r^2}{4},t_0))$ and
\EQ{\label{CKN-Prop1-eq1}
\norm{v}_{L^\infty(B_{\frac{r}{2}(x_0)} \times (t_0-\frac{r^2}{4},t_0))} \le \frac {C_{\CKN}\e^{\frac13}} {r} .
}
\end{lemma}

We recall a useful Gronwall-type inequality from  \cite[Lemma 2.2]{BT19}.
\begin{lemma}\label{gronwall}
Suppose $f \in L^\infty_{\loc}([0, T_0); [0,\infty))$
(which may be discontinuous) satisfies, for some
 $a$, $b>0$, and $m \ge 1$,%
$$
f(t) \le a + b
\int^t_0
(f(s) + f(s)^m)ds \qquad \mbox{for}\  t \in (0, T_0),
$$
then we have $f(t) \le 2a$ for $t \in (0, T)$
with
$$
T=\min\left(T_0,\, \frac{C}{b(1+a^{m-1})}\right),
$$
where $C$ is a universal constant.
\end{lemma}

Finally we also show the following elementary bound for the scaled energy.
\begin{lemma}
\label{lemmaA}

Assume that
 $f \in L^2_{\loc}(\R^3)$ and let $N_R=N_{(\alpha), R}(f):= 
 \displaystyle{\sup_{r \in (R, 1]}\frac{1}{r^\alpha}}\int_{B_r}|f|^2 <\infty$
 for some $\alpha>0$ and $R\in [0,\frac12]$. 
\begin{enumerate}
\item[(i)]
 If  $\delta \ge 2^\alpha N_R$, then for any $x_0 \in B_{\frac12}$
  we have
  \begin{equation}
\label{N_R}
\sup_{R(x_0)< r\le {1-|x_0|} }
\frac{1}{r^\alpha}\int_{B_r(x_0)}|f|^2
\le \delta
\end{equation}
with $R(x_0)=\max\bke{{\frac R2}, 2\bke{\frac{N_R}{\delta}}^{\frac1{\alpha}}|x_0| }$.

\item[(ii)]
 If $\alpha \ge 1$ and $\delta \ge 2 N_R$, 
then for any $x_0 \in B_{\frac12}$
  we have
  \begin{equation}
\sup_{R'(x_0)< r\le {1-|x_0|} }
\frac{1}{r}\int_{B_r(x_0)}|f|^2
\le \delta
\end{equation}
with $R'(x_0)=\max\bke{{\frac R2}, 
\frac{2^\alpha N_R}{\delta}|x_0|^{\alpha} }$.

\end{enumerate}
\end{lemma}
\begin{proof}~(i)
Assume that $r \in (R(x_0), 1-|x_0|]$. {If $1/2<r\le 1-|x_0|$,
\[
\frac1{r^{\alpha}}\int_{B_r(x_0)} |f|^2
\le \frac{1}{r^{\alpha}}\int_{B_{1}(0)} |f|^2
\le \frac{1}{r^{\alpha}} N_R \le \delta.
\]
If $|x_0|\le r\le 1/2$,}
\eqref{N_R} clearly holds by the following estimate:
$$
\frac1{r^{\alpha}}\int_{B_r(x_0)} |f|^2
\le \frac{1}{r^{\alpha}}\int_{B_{2r}(0)} |f|^2
\le 2^{\alpha}N_R \le \delta,
$$
where we have used $|x_0| +r \le 2r$. 
Finally, if $R(x_0)<\ r <|x_0|$ (this case is empty if $|x_0|\le R/2\le R(x_0)$), we have $|x_0|+r <2|x_0|$,
and hence
\begin{align*}
\frac1{r^{\alpha}}\int_{B_r(x_0)} |f|^2
&\le
\frac1{r^{\alpha}}\int_{B_{2|x_0|}(0)} |f|^2
\le
\bke{\frac{2|x_0|}{r}}^{\alpha} N_R.
\end{align*}
The right hand side is bounded by $\delta$ since
$r >R(x_0)\ge 2\bke{\frac{N_R}{\delta}}^{\frac1{\alpha}}|x_0|$.
Therefore we have verified \eqref{N_R}.%
\\
(ii) We note that if $R/2<r<|x_0|$,  
\begin{align*}
\frac1{r}\int_{B_r(x_0)} |f|^2
&\le
\frac1{r}\int_{B_{2|x_0|}(0)} |f|^2
\le
\frac{(2|x_0|)^{\alpha}}{r} N_R.
\end{align*}
Hence the right hand side is bounded by $\delta$ when
$r \ge \frac{2^\alpha N_R}{\delta} |x_0|^{\alpha}$. 
The other case is similar to the proof of (i), and we omit the details.
\end{proof}

\begin{remark}
\label{K_2}
In our previous paper \cite{KMT20}, we have used
the scale critical Herz norm in $\R^3$:
$$
\|f\|_{K^{-1+\frac3p}_{p,\infty}}:=\sup_{x_0 \neq 0} |x_0|^{1-\frac{3}{p}}
\|f\|_{L^p(B_{\frac{|x_0|}{2}}(x_0))}
$$
for $p\ge 3$. 
We easily observe that
$$
\|f\|_{K^{\frac12}_{2,\infty}}^2 \le C\sup_{r>0}\frac{1}{r}\int_{B_r}|f|^2dx \le C'\|f\|
_{K^{\frac12}_{2,\infty}}^2
$$
with some absolute constants $C$ and $C'>0$.

\end{remark}

\section{Main results}
\label{main results}

In this section we prove our main results, Theorems \ref{thm:main}-\ref{thm:0430}, regarding the regularity of the local energy solutions.
They are obtained as consequences of the following theorem.

\begin{theorem}
\label{large data}
Let $v$ be a local energy solution of \eqref{NS} in $\R^3 \times (0,T)$ associated with 
initial data  $v_0 \in L^2_{\uloc}$, and
\[
N_R=N_{(\alpha),R}(v_0):= 
\sup_{r \in (R, 1]} \frac{1}{r^\alpha}\int_{B_r(0)}|v_0|^2<\infty
\]
 for some $\alpha \in (1,4)$ and $R \in [0,\frac12]$.
There are absolute constants $c$ and $C>0$ such that the following holds.
\begin{enumerate}
\item[(i)]
If $N_R \le 1$ and $R \le c \sqrt{S_1}$ with 
$$
S_1:= \min \bke{T, 
\frac{c  }{1+\|v_0 \|_{L^2_{\uloc}}^{12} }
},
$$
 then we have 
\EQS{
 |v(x,t)|\le C \bkt{ 1 +  \|v_0 \|_{L^2_{\uloc}}  S_1^{\frac14(4-\al)} }
  t^{\frac14 (\alpha-3)}
\quad \text{ for } \  |x|\le \tfrac 12, \quad
c\max \bke{R^2, N_R ^{2/\al}|x|^2 } \le t  \le S_1.}
\item[(ii)]
If $N_R \ge 1$, assume 
 $R \le cN_R \sqrt{S_2}$ with
\EQ{\label{T2.def} 
S_2= \min \bke{T, 
\frac{c }{1+\|v_0 \|_{L^2_{\uloc}}^{12} },
  \bke{\frac{c}{N_R} }^{\frac{2\al+2}{\alpha -1}},
  \frac{c}{(1+\|v_0\|_{L^2_{\uloc}}^{4/3})N_R^2}},
}
 then we have 
\EQS{
 |v(x,t)|\le C \bkt{ N_R^{(\al+1)/2} + N_R^{3/2} \|v_0 \|_{L^2_{\uloc}}  S_2^{\frac14(4-\al)} }
  t^{\frac14 (\alpha-3)}
\\
\text{ for }\quad |x|\le \tfrac 12, \quad
\frac{c}{N_R^2}
 \max \bke{R^2, |x|^2 } \le t  \le S_2.
}
\end{enumerate}
\end{theorem}

The following local energy estimate is our key ingredient for the proof 
of Theorem \ref{large data}, which shows the smallness 
of the scaled local energy $E_r$ defined in \eqref{scaled energy}
is propagated at least locally in time
and some scales.

\begin{lemma}
\label{lem:main}
Let $v$ be a local energy solution of \eqref{NS} in $\R^3 \times (0,T)$ associated with 
initial data  $v_0 \in L^2_{\uloc}$, and
$N_R= \sup_{R< r \le 1} \frac{1}{r^\alpha}\int_{B_r(0)}|v_0|^2<\infty$
for  some constants $\alpha \in [1,4)$ and $R \in [0,1)$.
 Let $\delta \ge 70 N_R$. Then  for any $r \in (R,1]$ we  have
\begin{align}
E_r(t)
&\le \delta \qquad  \mbox{for} \ \ 0 < t \le \min 
\bke{ \lambda r^2, S_1},
\label{A_r}
\end{align}
where
$$\lambda=\frac{c}{1+\delta^2} \le 1 \quad \textit{and} 
 \quad S_1=  \min \left ( T, \,
 \frac{c\min\{1,\delta^4\} }{1+\|v_0\|_{L^2_{\uloc}}^{12}}
 \right).$$
Moreover, 
if $R<R_1:=\min \bke{\sqrt{\frac{S_1}{\lambda}},  1 }$,
there exists $c_2(t)\in \R$ such that
\begin{align}
 \frac{1}{r^2}
\int_0^{\lambda r^2}\!\!\int_{B_r}\abs{v}^3+|p-c_2(t)|^{\frac32}dxdt
&\le C(\lambda^{\frac14} \delta^{\frac32}r^{\frac32 (\alpha-1)}
+ \lambda \|v_0\|_{L^2_{\uloc}}^3r^{\frac{9}{2}}) \quad 
\mbox{for  all} \ r \in (R,R_1].
\label{02013}
\end{align}
Above $c$ and $C>0$ are absolute constants.
\end{lemma}

\begin{proof}
[Proof of Lemma \ref{lem:main}]
By the definition of the local energy solution,  for $r \in (R,1]$, there exists
$c_r=c_r(t)$
such that the pressure admits the following decomposition \eqref{decomposition}:
\begin{align}
p-c_r
&=-\frac{|v|^2}{3}+ p_{\loc}+p_{nonloc}
\notag
\\
&:=-\frac{|v|^2}{3}+ \mbox{p.v.}\int_{B_{3r}} K(x-y) (v\otimes v)(y)dy \label{pressure'}
\\
&\qquad+\int_{\R^3 \backslash B_{3r}} (K(x-y)-K(-y) )(v\otimes v)(y)dy
\notag
\end{align}
in $L^{3/2}(B_{2r}(0) \times (0,T))$. By considering \eqref{decomposition} in a larger region $B_2 \times (0,T)$,
we may remove the $r$-dependence of $c_r(t)$ by choosing $c_r(t)=c_{1}(t)$.
Since $(v, p-c_r)$ is a suitable weak solution to \eqref{NS} in 
$B_{2r}$, 
the local energy inequality \eqref{CKN-LEI} with the test function $\phi$
satisfying 
$\phi \in C^{\infty}_0(B_{2r})$ such that $0\le \phi \le 1$
in $B_{2r}$ with $\phi=1$ in $B_r$
and $\|\nabla^k \phi \|_{L^\infty} \le C_kr^{-k}$ 
readily yields
\begin{align}
E_r(t)
&\le
\frac{2^{\alpha}}{(2r)^\alpha}\int_{B_{2r}} |v_0|^2 
+
\frac{C}{r^{\alpha+2}}\int^t_0\! \int_{B_{2r}}|v|^2
+
\frac{C}{r^{\alpha+1}}\int^t_0\!\int_{B_{2r}}\!\!|v|^3+|v||p-c_r|.
\end{align}
By \eqref{pressure'} together with   
H\"older and Young inequalities, the last term involving the pressure can be estimated as follows:
\begin{align*}
\frac{1}{r^{\alpha+1}}\int^t_0\!\int_{B_{2r}} |v||p-c_r|
&\le 
(\frac{1}{r^{\frac{3\alpha}{2}+\frac12}}\int^t_0\!\int_{B_{2r}} |v|^3)^{\frac13}
(\frac{1}{r^{\frac{3\alpha}{4}+\frac54}}\int^t_0\!\int_{B_{2r}} |p-c_r|^{\frac32})^{\frac23}
\\
&\le 
\frac{1}{r^{\frac{3\alpha}{2}+\frac12}}\int^t_0\!\int_{B_{2r}} |v|^3
+
\frac{C}{r^{\frac{3\alpha}{4}+\frac54}}\int^t_0\!\int_{B_{2r}} |p-c_r|^{\frac32}
\\
&\le 
\frac{1}{r^{\frac{3\alpha}{2}+\frac12}}\int^t_0\!\int_{B_{2r}} |v|^3
+
\frac{C}{r^{\frac{3\alpha}{4}+\frac54}}\int^t_0\!\int_{B_{2r}} 
|v|^3+|p_{loc}|^{\frac32}+|p_{nonloc}|^{\frac32}.
\end{align*}
Noting that $\frac{3\alpha}{4}+\frac54\le \alpha +1 \le \frac{3\alpha}{2}+\frac12$ for $\alpha \ge 1$,  we have
\begin{align}
E_r(t)
&\le
2^{\alpha}N_{R}
+
\frac{C}{r^{\alpha+2}}\int^t_0\! \int_{B_{2r}}|v|^2
+
\frac{C}{r^{\frac{3\alpha}{2}+\frac12}}
\int^t_0\!\int_{B_{2r}}\!\!|v|^3+|p_{loc}|^{\frac32}
+
\frac{C}{r^{\frac{3\alpha}{4}+\frac54}}\int^t_0\!\int_{B_{2r}} 
|p_{nonloc}|^{\frac32}.
\notag
\\
&=:2^{\alpha} N_{R}+I_{lin} +I_{nonlin}+ I_{ploc}+I_{pnonloc}.
\label{A_r'}
\end{align}
We divide the estimate into two cases.

\medskip
\noindent
{\underline{\bf Case I:~$R< r\le \frac16$}.}
This case is empty if $R>\frac16$.
For the simplicity of notation, 
let $\mathcal{E}_r(t):= \sup_{\rho \in (r, 1]} E_{\rho}(t)$.
We easily see
\begin{align}
I_{lin}
&\le
\frac{C}{R^2}\int^t_0
\mathcal{E}_R(s)ds.
\label{vlin}
\end{align}
By the interpolation and Young's inequalities, we also have
\begin{align}
I_{nonlin}
&\le
\frac{C}{r^{\frac{3\alpha}{2} +\frac12}} \int^t_0
\bke{\int_{B_{r}} \abs{ \nabla v}^2}^{\frac{3}{4}}\bke{\int_{B_{r}} \abs{v}^2}^{\frac{3}{4}}
 +
 \bke{\frac{1}{r}\int_{B_{r}} \abs{v}^2}^{\frac{3}{2}}ds  \notag
\\
&=
\frac{C}{r^{2}} \int^t_0
\bke{\frac{1}{r^{\alpha-2}}\int_{B_{r}} \abs{ \nabla v}^2}^{\frac{3}{4}}
\bke{\frac{1}{r^{\alpha}}\int_{B_{r}} \abs{v}^2}^{\frac{3}{4}}
 +
 \bke{\frac{1}{r^\alpha}\int_{B_{r}} \abs{v}^2}^{\frac{3}{2}}ds   \notag
\\
&\le \epsilon \mathcal{E}_R(t)
+
\frac{C_{\epsilon}}{R^2} \int^t_0 \mathcal{E}_{R}(s)^3
+ \mathcal{E}_{R}(s)^\frac32 ds.
\label{nonlin}
\end{align}
This estimate and the Calder\'on-Zygmund estimate 
give
\begin{align}
I_{ploc}
&\le
\frac{C}{r^{\frac{3\alpha}{2}+\frac12}} 
\int^t_0 \int_{B_{3r}}\!\!|v|^3
\notag
\\
&\le
\epsilon \mathcal{E}_R(t)
+
\frac{C_{\epsilon}}{R^2} \int^t_0 \mathcal{E}_{R}(s)^3
+ \mathcal{E}_{R}(s)^\frac32 ds.
\label{ploc'}
\end{align}
On the other hand, since $|x-y| \ge |y|/3$
for $x \in B_{2r}$ and $y \in \R^3 \backslash B_{3r}$, we see%
\begin{align}
|p_{nonloc}(x)|
&\le
\int_{\R^3 \backslash B_{3r}}
|K(x-y)-K(-y)||v(y)|^2dy
\notag
\\
&\le
Cr\int_{\R^3 \backslash B_{3r}}
\frac{1}{|x-y|^4} |v(y)|^2dy
\notag
\\
&\le
Cr\int_{\R^3 \backslash B_{3r}}
\frac{|v(y)|^2}{|y|^4} dy
\notag
\\
&\le
Cr\sum_{k=1}^{\lfloor -\log_2 r-1 \rfloor} \int_{B_{2^{k}r} \backslash B_{2^{k-1}r}}
\frac{|v(y)|^2}{|y|^4} dy
+
Cr\int_{\R^3 \backslash B_{1/4}}
\frac{|v(y)|^2}{|y|^4}dy
\notag
\\
&\le
Cr^{\alpha-3}
\sum_{k=1}^{\lfloor -\log_2 r-1 \rfloor}
\frac{1}{2^{(4-\alpha)k}}
\Big(
\frac{1}{(2^{k}r)^\alpha} \int_{B_{2^k r}} |v(y)|^2 dy
\Big)
+
CrE_{\uloc}
\notag
\\
&\le
Cr^{\alpha-3}
\mathcal{E}_R
+
Cr\|v_0\|_{L^2_{\uloc}}^2,
\label{pnonloc}
\end{align}
provided $t\le T_{\uloc}  =\min \big(T, \frac{c_0}{1+\|v_0\|_{L^2_{\uloc}}^4}\big)$ and $\alpha <4$, 
where we have used Lemma \ref{apriori uloc} in the last line.
We then obtain
\begin{align}
I_{pnonloc}
&\le \frac{C}{r^{\frac{3\alpha}4 +\frac54}}
\int^t_0 r^3 \bke{ r^{\frac32(\alpha-3)}\mathcal{E}_R^{\frac32}(s)
+ r^{\frac32}\|v_0\|_{L^2_{\uloc}}^3 } ds  \notag
\\
&\le
\frac{C}{r^{2}}\int^t_0 \mathcal{E}_R(s)^{\frac32}ds +
Ctr^{\frac{13}{4}-\frac{3\alpha}{4}} \|v_0\|^3_{L^2_{\uloc}}
\notag
\\
&\le
\frac{C}{r^2}\int^t_0 \mathcal{E}_R(s)^{\frac32} ds+
\frac{\delta}{10},
\label{pnonloc'}
\end{align}
provided $t \le \bke{ \frac{c\delta}{\|v_0\|_{L^2_{\uloc}}^3},T_{\uloc} }$ with a small absolute constant $c>0$. We have used 
$\al \in [1,4)$ in the second inequality.
Hence plugging \eqref{vlin}, \eqref{nonlin}, \eqref{ploc'}, and \eqref{pnonloc'}
into \eqref{A_r'} with $\e=1/6$,  we obtain
\begin{align*}
\sup_{R< r <1/6}E_r(t)
&\le
\frac{\delta}{3} + \frac13\mathcal{E}_R(t)
+\frac{C}{R^2} \int^t_0 \mathcal{E}_R(s)
+
\mathcal{E}_R(s)^3 ds
\end{align*}
for $\delta \ge 70N_R$.

\medskip
\noindent
{\underline{\bf Case II:~$\frac16\le r\le 1$}.}
In order to bound the right hand side of \eqref{A_r'},
we observe from Lemma \ref{apriori uloc} that
\begin{equation}
\label{est:uloc}
\sup_{1/3\le r\le 2} E_r(t) \le C E_{\uloc}(t) \le C\|v_0\|_{L^2_{\uloc}}^2
\qquad \mbox{holds} \
\mbox{for} \ t\le T_{\uloc}.
\end{equation}
This shows 
\begin{align*}
I_{lin}
&\le
C \int^t_0\sup_{1/3\le r\le2}
E_{r}(s) ds
\le
Ct \|v_0\|_{L^2_{\uloc}}^2,
\end{align*}
and hence if $t\le \bke{T_{\uloc}, \frac{c\delta}{\|v_0\|_{L^2_{\uloc}}^2}}$
with a suitable small constant $c>0$, we
have
\begin{align}
\label{vlin2}
I_{lin}&\le
\frac{\delta}{40}.
\end{align}
For the nonlinear term in \eqref{A_r'}, we have
\begin{align}
I_{nonlin}
& \le
C \int_0^t\int_{B_{2}} \abs{ v}^3
\notag
\\
& \le
C \bke{\int_0^t\int_{B_{2}} \abs{ \nabla v}^2}^{\frac34}
\bke{\int^t_0 \bke{\int_{B_{2}} \abs{v}^2}^3}^{\frac{1}{4}}
 +C\int^t_0\bke{\int_{B_{2}} \abs{v}^2}^{\frac{3}{2}}
\notag
\\
&\le
C(t^{\frac14} +t) \|v_0\|_{L^2_{\uloc}}^3,\label{est:0310}.
\end{align}
By the Calder\'on-Zygmund estimate it also implies
\begin{align}
\label{est:03102}
I_{ploc}&\le
C\int^t_0\!\! \int_{B_{3}}\!\!|v|^3
\le
C(t^{\frac14} +t) \|v_0\|_{L^2_{\uloc}}^3.
\end{align}
Thus the right hand sides in \eqref{est:0310} and \eqref{est:03102} are
 bounded by $\frac{\delta}{40}$, provided
$$
t\le \min \bke{ T_{\uloc}, \frac{c\delta^4}{\|v_0\|_{L^2_{\uloc}}^{12}},
\frac{c\delta}{\|v_0\|_{L^2_{\uloc}}^{3}} }
$$
with some small absolute constant $c>0$.
On the other hand, in the same way as in \eqref{pnonloc},
we have
\begin{align*}
|p_{nonloc}(x)|
\le
C\int_{\R^3 \backslash B_{1/2}}
\frac{|v(y)|^2}{|y|^4} dy
&\le
C\|v_0\|_{L^2_{\uloc}}^2,
\end{align*}
which implies
\begin{align*}
I_{pnonloc}
&\le
\frac{\delta}{40} \qquad \mbox{for} \  \ t \le \frac{c\delta}{\|v_0\|_{L^2_{\uloc}}^3}.
\end{align*}
Making use of these estimates in \eqref{A_r'}, we obtain
$$
\sup_{1/6\le r \le 1}E_r(t)
\le
\frac{\delta}{3}. 
$$

By choosing $c_1>0$ sufficiently small, we may take
$$
S_1:= \min \bke{T, \, \frac{c_1\min\{1, \delta^4 \}}{1+\|v_0\|_{L^2_{\uloc}}^{12}} }
\le \min \bke{ T_{\uloc}, \frac{c\delta}{\|v_0\|_{L^2_{\uloc}}^2},\frac{c\delta^4}{\|v_0\|_{L^2_{\uloc}}^{12}},
\frac{c\delta}{\|v_0\|_{L^2_{\uloc}}^{3}} }.$$
Summarizing the conclusions of the cases I and II, we have
\begin{align*}
\mathcal{E}_R(t)
&\le
\frac{\delta}{2}
+\frac{C}{R^2} \int^t_0 \mathcal{E}_R(s)
+
\mathcal{E}_R(s)^3 ds
\end{align*}
 for $t\le S_1$.
Applying Lemma \ref{gronwall}
we obtain
\EQ{
\label{A_R}
\mathcal{E}_R(t)
\le \delta \quad  \mbox{for} \ t \le \min \bke{ \lambda R^2, S_1 }, 
\quad \la= \frac c{1+\de^2}.
}
Since $\delta \ge 70N_R \ge 70N_r$
for any $r\in (R,1]$, we may replace $R$ by $r$ in \eqref{A_R}, and thus we have
verified \eqref{A_r} for $R < r \le 1$.%

\medskip

To prove \eqref{02013}, 
we see that if $\lambda r^2 \in (0,S_1]$,
\begin{align*}
&\frac1{r^{\frac{3\alpha}{2} +\frac12}}\int^{\lambda r^2}_0 \int_{B_{r}}|v|^3\\
&\le
\frac{C}{r^{\frac{3\alpha}{2} +\frac12}} \int_0^{\lambda r^2}\bke{\int_{B_{r}} \abs{ \nabla v}^2}^{\frac{3}{4}}\bke{\int_{B_{r}} \abs{v}^2}^{\frac{3}{4}}ds
 +
 \frac{C}{r^{\frac{3\alpha}{2} +\frac12}}
 \int^{\lambda r^2}_0\bke{\frac{1}{r}\int_{B_{r}} \abs{v}^2}^{\frac{3}{2}}ds
\\
&\le
\frac{C}{r^{\frac{3\alpha}{2} +\frac12}} \int_0^{\lambda r^2}\bke{\int_{B_{r}} \abs{ \nabla v}^2}^{\frac{3}{4}}ds
 \  r^{\frac{3\alpha}4}E_r(\lambda r^2)^{\frac34}
 +
 \frac{C}{r^{\frac{3\alpha}{2} +\frac12}}
 \int^{\lambda r^2}_0
r^{\frac32(\alpha-1)} \bke{\frac{1}{r^\alpha}\int_{B_{r}} \abs{v}^2}^{\frac{3}{2}}ds
\\
&\le
\frac{C}{r^{\frac{3\alpha}{2} +\frac12}} 
\bke{\int_0^{\lambda r^2}\int_{B_{r}} \abs{ \nabla v}^2ds}^{\frac34}
(\lambda r^2)^\frac14
 \  r^{\frac{3\alpha}4}E_r(\lambda r^2)^{\frac34}
 + C\lambda 
  E_r(\lambda r^2)^{\frac{3}{2}}
\\
&\le C(\lambda^\frac14 +\lambda)E_r(\lambda r^2)^{\frac32}.
\end{align*}
Hence taking $R_1:=\min\bke{\sqrt{\frac{S_1}{\lambda}},1}$,%
we have 
\begin{align}
\frac1{r^{\frac{3\alpha}{2}+\frac12}}\int^{\lambda r^2}_0 \int_{B_{r}}|v|^3
&\le
C \lambda^{\frac14} \delta^{\frac32} \qquad \mbox{for} \ r \in (R, R_1].
\label{1228}
\end{align}
By the Calder\'on-Zygmund estimate we also have 
\begin{align}
\frac1{r^{\frac{3\alpha}{2}+\frac12}}\int^{\lambda r^2}_0 \int_{B_{r}}|p_{loc}|^{\frac32}
&\le
C \lambda^{\frac14} \delta^{\frac32} \qquad \mbox{for} \ r \in (R, R_1]. 
\label{1229}
\end{align}
From \eqref{pnonloc} we have 
\begin{align}
\label{0201}
\frac{C}{r^{\frac32 \alpha+\frac12}}
\int^{\lambda r^2}_0 \int_{B_{2r}} |p_{nonloc}|^{\frac32}
&\le
\frac{C}{r^{\frac32 \alpha + \frac12}}
\int^{\lambda r^2}_0 
r^3 \bke{ r^{\frac32 (\alpha -3)}
\mathcal{E}_R(s)^{\frac32} +
r^{\frac32} \|v_0\|^3_{L^2_{\uloc}} }ds
\\
\notag
&\le
C\lambda( \delta^{\frac32}
+
r^{6-\frac32 \alpha}\|v_0\|_{L^2_{\uloc}}^3 ).
\end{align}
Combining this with \eqref{1228} and \eqref{1229} we obtain \eqref{02013}.
\end{proof}
\begin{proof}
[Proof of Theorem \ref{large data}]
By Lemma \ref{lemmaA} using $R\le 1/2$,
for  each $\delta \ge 2^{\alpha} N_R$ and $x_0 \in B_{1/2}$
\begin{equation*}
\sup_{R(x_0)< r\le \rho}
\frac{1}{r^\alpha}\int_{B_r(x_0)}|v_0|^2dx \le \delta, 
\quad {\rho = 1-|x_0|}, 
\quad 
R(x_0)=\max \bke{\frac{R}{2}, 2\bke{\frac{N_R}{\delta}}^{\frac1{\alpha}}|x_0| }.
\end{equation*}
Let
$v_{x_0}(x,t)=\rho v(x_0+\rho x, \rho^2 t)$, $p_{x_0}(x,t)=\rho^2 p(x_0+\rho x, \rho^2 t)$.  Since $(v_{x_0},p_{x_0})$ solves \eqref{NS} in $\R^3 \times (0,\rho^{-2}T)$, corresponding to $(v,p)$ in $\R^3 \times (0,T)$,
 and $1/2 \le \rho \le 1$, we have
\[
\|v_{x_0}(t=0)\|_{L^2_{\uloc}} \le C\|v_0\|_{L^2_{\uloc}},
\]
\[
N_{\rho^{-1} R(x_0)} (v_{x_0})=
\sup_{\rho^{-1} R(x_0)< r\le 1}
\frac{1}{r^{\alpha}}\int_{B_r(0)}|v_{x_0}(t=0)|^2 
= 
\sup_{R(x_0)< r\le \rho}
\frac{\rho^{\alpha-1}}{r^{\alpha}}\int_{B_r(x_0)}|v_0|^2 
\le \delta .
\]
This ensures the assumption of Lemma \ref{lem:main} for $v_{x_0}$. If we take $\de_{x_0} = 70 \de$, then
from \eqref{02013}, there exists $c'_{(x_0)}(t)$ such that
\[
\sup_{\rho^{-1} R(x_0)< r \le R_1' } \frac{1}{r^2}
\int_0^{\lambda r^2}\!\!\!\int_{B_r(0)}|v_{x_0}|^3+|p_{x_0}-c'_{(x_0)}(t)|^{\frac32}
\le 
C (\lambda^{\frac14} \delta^{\frac32}r^{\frac32 (\alpha-1)}
+ \lambda \|v_0\|_{L^2_{\uloc}}^3r^{\frac{9}{2}}).
\]
Here  
$$R_1'=\min \big(\sqrt{\frac{S_1'}{\lambda}},1\big)  \qquad
\mbox{with} \quad  
S_1'=\min \bke{\rho^{-2}T, \, \frac{c_1\min\{1, \delta^4 \}}{1+\|v_0\|_{L^2_{\uloc}}^{12}}}.
$$
This implies
\begin{align*}
\sup_{R(x_0) \le r \le \rho R_1' } 
 \frac{1}{\lambda r^2}
\int_0^{\lambda r^2}\!\!\!\int_{B_{\sqrt\lambda r}(x_0)}\abs{v}^3+|p-c_{(x_0)}(t)|^{\frac32}
\le C
(\lambda^{-\frac34} \delta^{\frac32}r^{\frac32 (\alpha-1)}
+ \|v_0\|_{L^2_{\uloc}}^3r^{\frac{9}{2}}).
\end{align*}
Denote
$$
\ep(x_0,r):=\frac{1}{r^2}
\int_0^{r^2}\!\!\!\int_{B_r(x_0)}
\abs{v}^3+|p-c_{(x_0)}(t)|^{\frac32}.
$$
For $\sqrt{\lambda}R(x_0) \le r \le \sqrt{\lambda} \rho R_1'$ we have
\EQS{\label{0220a}
\ep(x_0,r)&\le C
[\lambda^{-\frac34} \delta^{\frac32}(\la^{-1/2}r)^{\frac32 (\alpha-1)}
+ \|v_0\|_{L^2_{\uloc}}^3 (\la^{-1/2}r)^{\frac{9}{2}}]
\\
&\le  C(1+\delta^{\frac{3\alpha}2})\delta^{\frac32}r^{\frac32(\alpha -1)}
+C(1+\delta)^{\frac92}r^{\frac92}\|v_0\|_{L^2_{\uloc}}^3.
}

We now choose $r>0$ as
$$
r \le R_2:=\min \bke{\bke{\frac{c}{(1+\delta^{\alpha})\delta} }^{\frac{1}{\alpha -1}},\frac{c}{\|v_0\|_{L^2_{\uloc}}^{2/3}(1+\delta)}}
$$
with a small constant $c>0$ so that $\ep(x_0,r)$ is smaller than $\e_{\CKN}$ in Lemma \ref{CKN-Prop1}.
This enables us to apply Lemma \ref{CKN-Prop1} for
$x_0 \in B_{1/2}$ and  
$$
t=r^2 
\in 
\left( \la R(x_0)^2
, \,
\min \bke{\lambda(\rho R_1')^2, R_2^2}
\right]
$$ 
to see that $v$ is regular at $(x_0,t)$ and, by \eqref{0220a},
\EQS{\label{0220b}
|v(x_0,t)|
&\le C_{\CKN}\ep(x_0, t^{\frac12})^{\frac13}t^{-\frac12}
\\
& \le C(1+\delta^{\frac{\alpha}2})\delta^{\frac12}
 t^{\frac14 (\alpha-3)}
 +C(1+\delta)^{\frac32}\|v_0\|_{L^2_{\uloc}} t^{\frac14}
 \\
& = C\bkt{ (1+\delta^{\frac{\alpha}2})\delta^{\frac12} + (1+\delta)^{\frac32}\|v_0\|_{L^2_{\uloc}} t^{\frac14(4-\al)}}
 t^{\frac14 (\alpha-3)}.
}

Since 
$1/2 \le \rho \le 1$, we may
take $c>0$ so small  that   
\begin{align*}
\lambda (\rho R_1')^2=\rho^2\min( S_1', \la)
&= \rho^2  
\min \bke{ \rho^{-2}T, \,  
\frac{c\min\{1,\delta^{4}\} }{1+\|v_0 \|_{L^2_{\uloc}}^{12} }, 
\, 
\la
}
\\
&\ge  
\min \bke{ T, 
\,
\frac{c\min\{1,\delta^4\} }{1+\|v_0 \|_{L^2_{\uloc}}^{12} }, 
\,
\frac{c}{(1+\delta)^2}
}.
\end{align*}
Thus \eqref{0220b} is valid for $x_0 \in B_{1/2}$ and $S_a \le t \le 
S_b$, where
\EQS{
\label{0220c}
S_a
&: = \frac{c}{1+\de^2}
 \max \bke{R^2, C\bke{{N_R}/{\delta}}^{\frac2{\alpha}}|x_0|^2 },
\\
S_b&:= \min \bke{T, 
\frac{c\min\{1,\delta^4\} }{1+\|v_0 \|_{L^2_{\uloc}}^{12} },
  \bke{\frac{c}{(1+\delta^{\alpha})\delta} }^{\frac{2}{\alpha -1}},
  \frac{c}{(1+\|v_0\|_{L^2_{\uloc}}^{4/3})(1+\delta)^2}}.
}

If $N_R \ge 1$, we can take $\de=70N_R$. Then
\[
S_a \le \frac{c}{N_R^2}
 \max \bke{R^2, |x_0|^2 },
\]
\[
S_b \ge  S_2= \min \bke{T, 
\frac{c }{1+\|v_0 \|_{L^2_{\uloc}}^{12} },
  \bke{\frac{c}{N_R} }^{\frac{2\al+2}{\alpha -1}},
  \frac{c}{(1+\|v_0\|_{L^2_{\uloc}}^{4/3})N_R^2}},
  \]
and
\EQS{
 |v(x_0,t)|\le C \bkt{ N_R^{(\al+1)/2} + N_R^{3/2} \|v_0 \|_{L^2_{\uloc}}  S_2^{\frac14(4-\al)} }
  t^{\frac14 (\alpha-3)}
\\
\text{ for }\quad
\frac{c}{N_R^2}
 \max \bke{R^2, |x_0|^2 } \le t  \le S_2.
}

If $N_R \le 1$, we can take $\de=70$. Then
\[
S_a \le c
 \max \bke{R^2, N_R ^{2/\al}|x_0|^2 },
\]
\[
S_b \ge  S_1= \min \bke{T, 
\frac{c  }{1+\|v_0 \|_{L^2_{\uloc}}^{12} }
},
\]
(This $S_1$ may have a smaller $c$ than $S_1$ in Part (i)), 
and
\EQS{
 |v(x_0,t)|\le C \bkt{ 1 +  \|v_0 \|_{L^2_{\uloc}}  S_1^{\frac14(4-\al)} }
  t^{\frac14 (\alpha-3)}
\\
\text{ for }\quad
c\max \bke{R^2, N_R ^{2/\al}|x_0|^2 } \le t  \le S_1. 
}
\end{proof}

\begin{proof}
[Proof of Theorem \ref{thm:main}]
\mbox{}\nopagebreak

\medskip
\noindent
For the proof of \eqref{0204}, we first claim 
that there exist $\epsilon_*$ and $c>0$ such that
if $N_{0}:=N_{(\alpha),0}:= \sup_{r\in (0,1]} 
\frac{1}{r^{\alpha}}\int_{B_r} \abs{v_0(x)}^2dx < \epsilon_*$, 
then
for any $x_0 \in B_{\frac12}$ and 
$r \in (cN_0|x_0|^{\alpha}, {c(1-|x_0|)R_1'}]$,
\begin{equation}
\label{ckn'}
\frac{1}{r^2}\int_0^{r^2}
\!\!\!\int_{B_r(x_0)}\abs{v}^3 +\abs{p-c_{(x_0)}(t)}^\frac32 dxdt
\le
\epsilon_{\CKN}
\end{equation}
holds for some $c_{(x_0)}(t) \in \R$ and $R_1'>0$ defined below. 
Here $\epsilon_{\CKN}$ is the small constant 
in Lemma \ref{CKN-Prop1}.
We observe from Lemma \ref{lemmaA} (ii)
that for  each $\eta \ge 2 N_0$ and $x_0 \in B_{\frac12}$
\begin{equation*}
\sup_{R(x_0)< r\le \rho}
\frac{1}{r}\int_{B_r(x_0)}|v_0|^2dx \le \eta, 
\quad {\rho = 1-|x_0|}, 
\quad 
R(x_0)=  \frac{2^\alpha 
N_0}{\eta}|x_0|^{\alpha} .
\end{equation*}
Let
$v_{x_0}(x,t)=\rho v(x_0+\rho x, \rho^2 t)$, $p_{x_0}(x,t)=\rho^2 p(x_0+\rho x, \rho^2 t)$ and $\delta=70\eta$.  
As in the proof of Theorem \ref{large data} we have
\[
\|v_{x_0}(t=0)\|_{L^2_{\uloc}} \le C\|v_0\|_{L^2_{\uloc}},
\]
\[
\sup_{\rho^{-1} R(x_0)< r\le 1}
\frac{1}{r}\int_{B_r(0)}|v_{x_0}(t=0)|^2 
= 
\sup_{R(x_0)< r\le \rho}
\frac{1}{r}\int_{B_r(x_0)}|v_0|^2 
\le
\eta 
=
\frac \de {70}.
\]
This guarantees the assumption of Lemma \ref{lem:main} for $v_{x_0}$ and $\alpha=1$, and hence there exists $c'_{(x_0)}(t)$ such that 
\[
\sup_{\rho^{-1} R(x_0)< r \le R_1' } \frac{1}{r^2}
\int_0^{\lambda r^2}\!\!\!\int_{B_r(0)}|v_{x_0}|^3+|p_{x_0}-c'_{(x_0)}(t)|^{\frac32}
\le C\lambda^{\frac14}\delta^{\frac32}.
\]
Here
$$R_1'=\min \bke{\sqrt{\frac{S_1'}{\lambda}},1}  \qquad
\mbox{with} \quad  
S_1'=\min \bke{\rho^{-2}T, \, \frac{c_1\min\{1, \delta^4 \}}{1+\|v_0\|_{L^2_{\uloc}}^{12}}}.
$$
This implies
\begin{align}
\sup_{R(x_0) \le r \le \rho R_1'} 
\frac{1}{\lambda r^2}
\int_0^{\lambda r^2}\!\!\!\int_{B_{\sqrt\lambda r}(x_0)}\abs{v}^3+|p-c_{(x_0)}(t)|^{\frac32}
\le \frac{C\delta^{\frac32}}{\lambda^{\frac{3}4}}
\le C(1+\delta^{\frac32})\delta^{\frac32}.
\end{align}
Take a constant $\delta_0>0$ so small that $C(1+\delta^{\frac{3}2}_0)\delta_0^{\frac32}\le \epsilon_{\CKN}$ and
 assume that $v_0$ satisfies $N_0 \le \delta_0/70$. 
Then we may choose $\delta=\delta_0$ since 
$\delta_0 \ge 70N_0$.
With this choice and with 
$\lambda_0=\lambda(\delta_0)$, we have verified
\eqref{ckn'}  for
$\lambda_0^{\frac12} R(x_0) < r \le \lambda_0^{\frac12}\rho R_1'$.

This enables us to apply Lemma \ref{CKN-Prop1} for
$x_0 \in B_{1/2}$ and  $t_0=r^2 \in \left(C \lambda_0
 \bke{\frac{N_0}{\delta_0}}^{2}|x_0|^{2\alpha} , \lambda_0
(\rho R_1')^2\right]$ 
to see
$$
|v(x_0,t_0)|\le C_{\CKN}r^{-1} 
= C_{\CKN}t_0^{-\frac12},
$$
and hence $v$ is regular at $(x_0,t_0)$. Since %
$\la_0$ and $\delta_0$ are absolute constants
and $1/2 \le \rho \le 1$, we may
choose  $c>0$ so small  that   
\begin{align*}
\lambda_0 (\rho R_1')^2 =\rho^2\min( S_1', \la_0)
&= \rho^2  
\min \bke{ \rho^{-2}T, \,  
\frac{c\min\{1,\delta_0^{4}\} }{1+\|v_0 \|_{L^2_{\uloc}}^{12} }, \, \la_0 }
\notag
\\
&\ge  
\min \bke{ T, \,  \frac{c}
{1+\|v_0\|_{L^2_{\uloc}}^{12}
}}.
\end{align*}
Thus 
$|v(x_0,t_0)|\le C_{\CKN} t_0^{-\frac12}$ for $x_0 \in B_{1/2}$ and 
$
 cN_0^{2}|x_0|^{2\al}  \le t _0\le \min\bke{T,  \,  \frac{c}{1+\|v_0\|_{L^2_{\uloc}}
^{12 }}}.
$
This proves \eqref{0204} for the case $N_0 \le \ep_*$.

In order to consider the case $N_0 > \ep_*$, let
$u(x,t)=\rho v(\rho x, \rho^2 t)$, $q(x,t)=\rho^2 p(\rho x, \rho^2 t)$,
$u_0(x)=\rho v_0(\rho x)$
with $0<\rho < 1$ to be given below.
We easily see 
\[
\|u_0\|_{L^2_{\uloc}}
=\rho^{-\frac12} \sup_{x_0\in \R^3}\, \|v_0\|_{L^2(B_\rho(x_0))} 
 \le C\rho^{-\frac12}\|v_0\|_{L^2_{\uloc}},
\]
\begin{equation}
\label{0304}
N_{0}(u_0)
=
\sup_{r \in (0,1]}
\frac{1}{r^{\al}}\int_{B_r}|u_0|^2 
=
\sup_{r \in (0,1]}
\frac{1}{\rho r^{\al}}\int_{B_{\rho r}}|v_0|^2 
=
\sup_{r \in (0,\rho]}
\frac{\rho^{\al -1}}{r^{\al}}\int_{B_r}|v_0|^2 
\le
\rho^{\al -1} N_0.
\end{equation}
Let $\rho =\bke{\frac{\epsilon_*}{N_0}}^{\frac{1}{\al -1}}$ 
so that $N_{0}(u_0) \le \epsilon_*$.
Then it follows from Step 1 that $u$ is regular and satisfies 
$|u(x,t)|\le C_{\CKN} t^{-\frac12}$ in $\Pi$
in the region
\[
\bket{ (x,t) \in B_{\frac12} \times (0,\rho^{-2}T); \ 
 cN_{0}(u_0)^{2}|x|^{2\al}  \le t \le \frac{c}{1+\|u_0\|_{L^2_{\uloc}}
^{12 }}}.
\]
Rescaling back to $(v, p)$ and using \eqref{0304}, we see that $v$ is regular  and satisfies $|v(x,t)|\le C_{\CKN} t^{-\frac12}$ in the region
 \[
\bket{ (x,t) \in B_{\frac{\rho}2} \times (0,T); \ 
 cN_0^2|x|^{2\al}  \le t \le \frac{c\rho^2}{1+\rho^{-6}\|v_0\|_{L^2_{\uloc}}
^{12 }}}.
\]
By taking $c>0$ sufficiently small, we see that 
this region contains
\EQ{\label{T1.def}
\Pi_1 = \bket{ (x,t) \in B_{\frac12} \times (0,T); \ 
 cN_0^2|x|^{2\al} 
\le t 
\le T_1}, \quad T_1:=\frac{c}{
N_0^{\frac{2}{\alpha -1}}
(1+N_0^{\frac{6}{\alpha -1}}\|v_0\|_{L^2_{\uloc}}^{12})}.
}
This completes the proof of \eqref{0204}.

\medskip

Est.~\eqref{0205} is the special case of Theorem \ref{large data}
with $R=0$. Indeed for all $N_0<\infty$ we have
\EQ{\label{0304a}
|v(x,t)| \le M t^{\frac{\al-3}4} \quad (C_1 |x|^2 < t< S_2),
}
where $M =C \bkt{ 1 +  \|v_0 \|_{L^2_{\uloc}}S_1^{\frac14(4-\al)}}$ if $N_0 \le 1$, and  $M = C \bkt{ N_0^{(\al+1)/2} + N_0^{3/2} \|v_0 \|_{L^2_{\uloc}}  S_2^{\frac14(4-\al)} }$ if $N_0 \ge 1$. Hence we can take
\[
C_0 =  C \bkt{1+ N_0^{(\al+1)/2} + (1+ N_0^{3/2}) \|v_0 \|_{L^2_{\uloc}}  S_1^{\frac14(4-\al)} }.
\]
We also have $C_1 = cN_0^{2/\al}$ if $N_0 \le 1$, and  $C_1 = cN_0^{-2}$ if $N_0 \ge 1$. Hence we can take
\[
C_1 = \frac {c N_0^{2/\al}}{(1+N_0)^{2+2/\al}}.
\]
Therefore \eqref{0304a} implies \eqref{0205}
of Theorem \ref{thm:main}.
This completes the proof.
\end{proof}

\begin{proof}
[Proof of Theorem \ref{thm:0209}]
The proof is similar to that of \eqref{0205}, but we give the details for completeness.
We first observe that  
\begin{equation}
\label{0211}
N(x_0):= \sup_{r\in (0,1]} \frac{1}{r^{\alpha}}
\int_{B_r(x_0)}|v_0|^2
 \le C\|v_0\|_{L^q}^2 <\infty
\end{equation}
 for $x_0\in B_1$ with $\alpha =3-\frac{6}{q} \in (1,3]$ for $3<q\le \infty$. Here 
 we write $\|\cdot\|_{L^q}=\|\cdot\|_{L^q(B_2)}$. 
For each $x_0 \in B_1$ we may adapt the proof of 
\eqref{0205} for $v(x+x_0,t)$ and $R=0$. 
Indeed for any $\delta \ge 70N(x_0)$ restricting the time variable
in \eqref{0220b} and \eqref{0220c} as 
$\widetilde{S}_a \le t \le \widetilde{S}_b$ with
\EQN{
\widetilde{S}_a
&: = \frac{c}{1+\de^2}\bke{\frac{N(x_0)}{\delta}}^{\frac2{\alpha}}|x-x_0|^2,
\\
\widetilde{S}_b
& :=\min \bke{S_b,\, 
\frac{1}{1+\delta^2}\bke{
\frac{c\delta}{\|v_0\|_{L^2_{\uloc}}}}^{\frac{4}{4-\alpha}} },
 }
we have 
\EQ{
|v(x,t)|
 \le  C\bkt{ (1+\delta^{\frac{\alpha}2})\delta^{\frac12} + (1+\delta)^{\frac32}\|v_0\|_{L^2_{\uloc}} t^{\frac14(4-\al)}}
 t^{\frac14 (\alpha-3)}
 \le C(1+\delta^{\frac{\alpha}2})\delta^{\frac12}
 t^{\frac14 (\alpha-3)}.
}
The second inequality is by the second argument of $\widetilde{S}_b$.
By \eqref{0211}, taking $C>0$ sufficiently big, 
we may choose $\delta =C\|v_0\|_{L^q}^2$ uniformly on 
$x_0 \in B_1$. Taking $x=x_0$, this implies
\begin{align*}
\|v(t)\|_{L^\infty(B_1)}
 \le 
 C(1+\|v_0\|_{L^q}^{\alpha})\|v_0\|_{L^q}
 t^{\frac14 (\alpha-3)}
\end{align*}
for $0<t<T_3$ where
\begin{multline}
T_3=
\min \Bigg( T, 
\,
\frac{c\min \bke{1, \|v_0\|_{L^q}^8}}{1+\|v_0 \|_{L^2_{\uloc}}^{12} }, 
\,
\bke{\frac{c}
{(1+\|v_0\|_{L^q}^{\alpha})\|v_0\|_{L^q}} }^{\frac{4}{\alpha -1}},\\
\,
 \frac{c}{(1+\|v_0\|_{L^2_{\uloc}}^{4/3})(1+\|v_0\|_{L^q}^4)}, 
\,
\frac{1}{1+\|v_0\|_{L^q}^4}\bke{
\frac{c\|v_0\|_{L^q}^2}{\|v_0\|_{L^2_{\uloc}}}}^{ \frac{4}{4-\alpha} } \Bigg) .
\end{multline}
This shows the desired claim.
\end{proof}

\medskip

\begin{proof}[Proof of Theorem \ref{thm:0430}]
For $\alpha:=-\beta >1$ we observe that 
\[
\frac{1}{r^\alpha}\int_{B_r} \abs{v_0}^2
\le \int_{B_r} |x|^{\beta}|v_0|^2
\le
L:=\|v_0\|_{L^{2,\beta}(\R^3)}^2 <\infty.
\]
Hence $N_{(\al)}(v_0)= \sup_{r \in (0,1]}\frac{1}{r^\alpha}\int_{B_r} \abs{v_0}^2 \le L $. By \eqref{0204} of Theorem \ref{thm:main}, $v$ is regular and satisfies $|v(x,t)| \le C t^{-1/2}$ in the region
\EQ{\label{0319a}
\bket{(x,t) \in B_{\frac12} \times (0,\infty)\,:
 c_1L\abs{x}^{2\alpha}\le t<T_1}
}
with $T_1=T_1(\norm{v_0}_{L^2_\uloc}, L,\al)$, while $c_1$ and $C$ depend  only on $\al$.
Let $T_4= \min(T_1,c_1 L2^{-2\al})$. The region
\[
\bket{(x,t) \in \R^3 \times (0,\infty)\,:
 c_1L\abs{x}^{2\alpha}\le t<T_4}
\]
is a subset of \eqref{0319a}. This finishes the proof of Theorem \ref{thm:0430}.
\end{proof}

\end{document}